\documentclass[reqno,11pt]{amsart}
\usepackage{amssymb}
\usepackage{amsmath}
\usepackage{amsthm}
\usepackage[left=1.5cm, right=1.5cm, top=1.5cm, bottom=1.5cm]{geometry}

\newtheorem{Theorem}{Theorem } [section]
\newtheorem{lemma}[Theorem]{Lemma}
\newtheorem{corollary}[Theorem]{Corollary}

\newtheorem{remark}{Remark}

\numberwithin{equation}{section}

\DeclareMathOperator{\sgn}{sgn}
\DeclareMathOperator{\supp}{supp}
\DeclareMathOperator{\essup}{essup}

\begin{document}

\title{Periodic Cauchy Problem for one Two-dimensional Generalization\\
of the Benjamin-Ono Equation in Sobolev Spaces of Low Regularity }
\author{Eddye Bustamante, Jos\'e Jim\'enez Urrea and Jorge Mej\'{\i}a}

\address{Eddye Bustamante M., Jos\'e Jim\'enez Urrea, Jorge Mej\'{\i}a L. \newline
Departamento de Matem\'aticas\\Universidad Nacional de Colombia\newline
A. A. 3840 Medell\'{\i}n, Colombia}
\email{eabusta0@unal.edu.co, jmjimene@unal.edu.co, jemejia@unal.edu.co}

\begin{abstract}
In this work we prove that the initial value problem (IVP) associated to the two-dimensional Benjamin-Ono equation
$$\left. \begin{array}{rl} u_t+\mathcal H \Delta u +uu_x &\hspace{-2mm}=0,\qquad\qquad (x,y)\in\mathbb T^2,\; t\in\mathbb R,\\ u(x,y,0)&\hspace{-2mm}=u_0(x,y), \end{array} \right\}\,,$$
where $\mathcal H$ denotes the Hilbert transform with respect to the variable $x$ and $\Delta$ is the Laplacian with respect to the spatial variables $x$ and $y$, is locally well-posed in the periodic Sobolev space $H^s(\mathbb T^2)$, with $s>7/4$.
\end{abstract}

\maketitle
\section{Introduction}

In this article we consider the initial value problem (IVP) associated to the two-dimensional Benjamin-Ono (BO) equation
\begin{align}
\left. \begin{array}{rl}
u_t+ \mathcal H \Delta u +uu_x &\hspace{-2mm}=0,\qquad\qquad (x,y)\in\mathbb T^2,\; t\in\mathbb R,\\
u(x,y,0)&\hspace{-2mm}=u_0(x,y),
\end{array} \right\}\label{BO}
\end{align}
where
$$\begin{array}{rccc} u:&\mathbb T^2\times \mathbb R&\to &\mathbb R\\
&(x,y,t)&\mapsto &u(x,y,t),
\end{array}
$$
$\mathbb T^2:=\mathbb R^2/(2\pi \mathbb Z)^2$ is  the two-dimensional torus, $\mathcal H$ denotes the Hilbert transform with respect to the variable $x$, which is defined for $2\pi$-periodic functions $f$ in $\mathbb T^2$ such that for a.e. $y\in\mathbb T$ $\int_{\mathbb T}f(x,y)dx=0$ through the Fourier coefficients by
\begin{align}
\left\{ \begin{array}{l} (\mathcal H f)^\wedge (m,n):=-i \sgn(m) \widehat f (m,n),\quad (m,n)\in \mathbb Z^2\quad \text{and} \quad m\neq 0\\
(\mathcal H f)^\wedge (0,n):=0,\quad n\in\mathbb Z
\end{array}
\right. \label{hilbert}
\end{align}
and $\Delta$ is the two-dimensional Laplacian in $\mathbb T^2$ defined by
\begin{align} 
\widehat{\Delta f}(m,n):=-(m^2+n^2)\widehat f (m,n)\quad\text{for}\quad(m,n)\in\mathbb Z^2\,.
\end{align}

The equation in \eqref{BO}, called Shrira equation, is a two-dimensional generalization of the BO equation
\begin{align}
u_t+\mathcal H u_{xx}+u u_x=0,\label{I1-4}
\end{align}
and was deduced by Pelinovsky and Shrira in \cite{PS1995} as a model for the propagation of long weakly nonlinear two-dimensional waves in deep stratified fluids. Very recently, Esfahany and Pastor in \cite{EP2018} studied for this equation existence, regularity and decay properties of solitary waves.\\

Using the abstract theory developed by Kato in \cite{Ka1975} and \cite{Ka1979}, it can be established the local well-posedness (LWP) of the IVP \eqref{BO} in $H^s(\mathbb T^2)$, with $s>2$. Nevertheless this approach ignores the dispersive effects of the linear part of the equation in \eqref{BO}.\\

The Cauchy problem for the one-dimensional BO equation \eqref{I1-4} has been extensively studied on the real line and in the periodic setting.\\

On the real line, using the dispersive character of the linear part of the equation, global well-posedness (GWP) of the IVP for the BO equation \eqref{I1-4} has been established in $H^s(\mathbb R)$, for $s=\frac32$ by Ponce in \cite{P1991} and LWP was proved for $s>\frac54$ in \cite{KT2003} by Koch and Tzvetkov. In \cite{KeKoe2003}, based on ideas of Koch and Tzvetkov in \cite{KT2003}, Kenig and Koenig obtained a refined version of the Strichartz estimate, which allowed them to establish LWP of the Cauchy problem in $H^s(\mathbb R)$ with $s>\frac98$. In \cite{LPS2014} Linares, Pilod and Saut, studying a family of fractional KdV equations, obtained the same result of Kenig and Koenig. Using an appropriate gauge transformation Tao in \cite{T2004} established GWP of the Cauchy problem for the BO equation \eqref{I1-4} in $H^1(\mathbb R)$. Following the Tao's approach, Burq and Planchon in \cite{BP2008} and Ionescu and Kenig in \cite{IK2007}, obtained GWP in $H^s(\mathbb R)$, $s>0$, and $L^2(\mathbb R)$, respectively.\\

In the periodic setting, using standard compactness arguments, Iorio in \cite{I1986} proved LWP of the Cauchy problem for the unidimensional BO equation in $H^s(\mathbb T)$, for $s>\frac32$. In \cite{MR2008}, Molinet and Ribaud, by using the gauge transformation introduced by Tao in \cite{T2004} and Strichartz estimates, established GWP in $H^1(\mathbb T)$. In \cite{M2007}, with the aproach in \cite{MR2008} and estimates in Bourgain type spaces, Molinet proved GWP of the Cauchy problem for the unidimensional BO equation \eqref{I1-4} in the energy space $H^{1/2}(\mathbb T)$. This latter result was improved by Molinet in \cite{M2008}, where he established GWP in $L^2(\mathbb T)$.\\

For the two-dimensional BO equation in $\mathbb R^2$ in \cite{BJM2019} we established LWP in $H^s(\mathbb R^2)$ with $s>\frac32$, where the main ingredient was a Strichartz estimate, similar to that obtained by Kenig in \cite{K2004}. In  \cite{BJM2019} we followed the same approach used by Linares, Pilod and Saut in \cite{LPS2014} for dispersive perturbations of Burger's equation and in \cite{LPS2017} for fractional Kadomtsev-Petviashvili equations.\\

Inspired by the works \cite{IK20071} and \cite{LPRT}, in this paper we consider the two-dimensional BO equation in the periodic setting. The statement of our result is as follows.

\begin{Theorem}\label{TP} Let $s>7/4$. Then, for every $u_0\in H^s(\mathbb T^2)$ such that
$$\int_0^{2\pi} u_0(x,y) dx=0\quad \text{a.e. $y\in \mathbb T$},$$
there exist a positive time $T:=T(\|u_0\|_{H^s})$ and a unique solution $u\in C([0,T];H^s(\mathbb T))$ to the IVP \eqref{BO} such that $u,u_x,u_y\in L^1([0,T];L^\infty(\mathbb T^2))$.\\
Moreover, for any $T'$, $0<T'<T$, there exists a neighborhood $U$ of $u_0$ in $H^s(\mathbb T^2)$ such that the flow map datum-solution
\begin{align*} S^s_{T'}: U\cap \left\{v_0\mid \text{for a.e.}\;y\int_{\mathbb T}v_0(x,y)=0 \right\} &\to C([0,T'];H^s(\mathbb T^2))\\
v_0&\mapsto v,
\end{align*}
is continuous.
\end{Theorem}

\begin{remark} The definition of the periodic Sobolev spaces $H^s(\mathbb T^2)$ is given in section 2.
\end{remark}

The proof of Theorem \ref{TP} follows the ideas of Ionescu and Kenig in \cite{IK20071} for the periodic KP-I equations. It uses in a crucial way a time-frequency localized Strichartz estimate (see Lemma \ref{L3-1} below). This estimate allows us to overcome the lack of Sobolev embedding when we are working with low regularity initial data. It is important to point out that in the periodic setting we do not have Strichartz estimates similar to those obtained in $\mathbb R^2$. Obtaining this local Strichartz estimate is based on the unidimensional Poisson Summation formula (see Lemma \ref{L3-2} below) and the Weyl's inequality (see Lemma \ref{weyl} below), which permits to bound sums of the form $\displaystyle{\sum_{m=1}^N e^{2\pi i f(m)}}$, where $f$ is a polynomial of degree greater than or equal to 2. The need to combine Lemmas \ref{L3-2} and \ref{weyl}, unlike what happens in the case of the Zakharov-Kuznetsov equation (see \cite{LPRT}) in which it suffices to apply the two-dimensional Poisson summation formula, arises from the fact that in our case the symbol $e^{-it \sgn(m)(m^2+n^2)}$ is not an smooth function in $\mathbb R^2_{m,n}$.\\

The Corollary \ref{C3-5} of the localized Strichartz estimate (Lemma \ref{L3-1}) is fundamental in order to control the norm $\|w\|_{L^1_TL^\infty}+\|\nabla\| _{L^1_TL^\infty}$ of solutions of the IVP \eqref{BO}, corresponding to smooth initial data (see Lemma \ref{gT}). The proof of Lemma \ref{gT} also uses an estimate for the norm $H^s(\mathbb T^2)$ of the product of periodic functions, (see Lemma \ref{Pro}), which we have proved from a similar estimate in $H^s(\mathbb R^2)$.\\

On the other hand, in order to obtain the energy estimate, contained in Lema \ref{EE}, it is necessary to use a Kato-Ponce commutator estimate in the periodic context (see Lemma \ref{DC}), which was proved in \cite{LPRT}.\\

The a priori estimates of Lemmas \ref{EE} and \ref{gT} allow us to use the Bona-Smith compactness method because from these estimates it is possible to obtain a common time $T>0$, where all the approximate solutions are defined, whose limit is the solution of the IVP \eqref{BO}.\\

This article is organized as follows: section 2 is devoted to explain basic definitions and notation; in section 3 we establish a time-frequency localized Strichartz estimate in the periodic case (Lemma \ref{L3-1}), one of the main ingredients in the proof of Theorem \ref{TP}. In section 4, we use a Kato-Ponce's commutator inequality (see Lemma \ref{DC}) and a Product Lemma (see Lemma \ref{Pro}) in the periodic context, to establish two a priori estimates (see Lemmas \ref{EE} and \ref{gT}) for sufficiently smooth solutions of the two-dimensional BO equation. Finally, in section 5, we use the Bona-Smith argument to establish the existence of solution of the IVP \eqref{BO}.

\section{Notation}

In this section we summarize our basic definitions and notation.

\subsection{} For $s\geq 0$, the periodic Sobolev space $H^s(\mathbb T^2)$ is defined by
$$H^s(\mathbb T^2):=\{g\in L^2(\mathbb T^2): \|g\|_{H^s}:=\|\widehat g(m,n)(1+m^2+n^2)^{s/2}\|_{l^2(\mathbb Z^2)}<\infty\},$$
and the space $H^\infty(\mathbb T^2)$ is defined by
$$H^\infty(\mathbb T^2):=\bigcap_{s\geq 0} H^s(\mathbb T^2).$$

\subsection{} For $s\geq 0$, the operator $J^s_{\mathbb T^2}$ on $S'(\mathbb T^2)$ (tempered distributions on $\mathbb T^2$) is defined by
$$\widehat{J^s_{\mathbb T^2}g}(m,n)=(1+m^2+n^2)^{s/2}\widehat g(m,n),\quad (m,n)\in\mathbb Z^2. $$

\subsection{} For $f:\mathbb T^2\times [0,T]\to\mathbb R$ the notation $\|f\|_{L^p_TL^q_{xy}}$ means
$$\left( \int_0^T \left( \int_0^{2\pi} \int_0^{2\pi} |f(x,y,t)|^q dx dy \right)^{p/q} dt \right)^{1/p}.$$
When $p=\infty$ or $q=\infty$ we must do the obvious changes with $\essup$.

\subsection{} In general, for a certain Banach space $X$ and $f:[0,T]\to X$, the notation $\|f\|_{L^p_TX}$ or $\|f\|_{L^p([0,T];X)}$ means
$$\left( \int_0^T \|f(t)\|_X^p dt \right)^{1/p}.$$

\subsection{} For any set $A$, we denote by $\chi_A$ the characteristic function of $A$.

\subsection{} Following the reference \cite{IK20071} for $k=0,1,2,\dots$ we define the operators $Q_x^k,Q_y^k,\widetilde{Q_x^k}$, and $\widetilde {Q_y^k}$ on $S'(\mathbb T^2)$ as follows:
\\

Given $g\in S'(\mathbb T^2)$ we have:
\begin{align*}
\widehat{Q_x^kg}:=&\chi_{[2^{k-1},2^k)}(|m|)\widehat g(m,n)\quad \text{if $k\geq 1$, and}\\
\widehat{Q_x^0g}:=&\chi_{[0,1)}(|m|)\widehat g(m,n);\\
\widehat{Q_y^kg}:=&\chi_{[2^{k-1},2^k)}(|n|)\widehat g(m,n)\quad \text{if $k\geq 1$, and}\\
\widehat{Q_y^0g}:=&\chi_{[0,1)}(|n|)\widehat g(m,n);\\
\widetilde{Q_x^k}:=&\sum_{k'=0}^k Q_x^{k'};\\
\widetilde{Q_y^k}:=&\sum_{k'=0}^k Q_y^{k'}.
\end{align*}

\subsection{} For $N\in P:=\{0,1,2,4,\dots,2^k,\dots\}$ we define the operator $\widetilde{P_N}$ on $S'(\mathbb T^2)$ by
\begin{align*}
\widetilde{P_N}:=&Q_x^0 Q_y^0\quad \text{if $N=0$};\\
\widetilde{P_N}:=&\widetilde{Q_x^k} Q_y^k+\widetilde{Q_y^{k-1}}Q_x^k\quad \text{if $N=2^k$ for some $k\in\mathbb N\cup \{0\}$}.
\end{align*}

\subsection{} According to the definitions in 2.7 it can be seen that the norm $\|g\|_{H^s(\mathbb T^2)}$ is equivalent to the norm
\begin{align}
\left( \sum_{N\in P} (1 \lor N)^{2s} \|\widetilde {P_N} g\|_{L^2(\mathbb T^2)}^2\right)^{1/2},\label{equiv_norm_g}
\end{align}
where $1\vee N$ denotes the maximum between 1 and $N$.

\subsection{} We will denote the Fourier transform and its inverse by the symbols $\mathcal F$ and $\mathcal F^{-1}$, respectively.

\subsection{} For variable expressions $A$ and $B$ the notation $A\lesssim B$  and the notation $A\gtrsim B$ mean that there exists a  universal positive constant $C$ such that $A\leq CB$ and $A\geq CB$, respectively, and the notation 
$A\sim B$ means that there exist universal positive constants $c$ and $C$ such that $cA\leq B\leq CA$.

\section{Linear Estimates}

This section is dedicated to the demonstration of a local version of the Strichartz estimate satisfied by the group associated with the linear part of the two-dimensional BO equation in the periodic case (see Lemma \ref{L3-1}). As a result of this estimate we also prove a boundedness property of the norm $L^1_TL^{\infty}(\mathbb T^2)$ of the periodic smooth solutions of one linear inhomogeneous BO equation (see Corollary \ref{C3-5}), which we will use later in the proof of Lemma \ref{gT}, essential to guarantee a common time interval for all approximated solutions of the IVP \eqref{BO}.\\

Let $\{W(t)\}_{t\in\mathbb R}$ be the group in $H^s(\mathbb T^2)$, associated to the linear part of the equation in \eqref{BO}, defined by
\begin{align}
[W(t)u_0](\overline x)=\frac1\pi \sum_{m\in\mathbb Z-\{0\}} \sum_{n\in\mathbb Z}e^{ i [\overline m \cdot\overline x-t  \sgn(m)(m^2+n^2)]} \widehat u_o(\overline m),\label{group}
\end{align}
where $\overline m:=(m,n)$, $\overline x:=(x,y)\in\mathbb T^2$,
$$\widehat u_0(\overline m):=c\int_0^{2\pi}\int_0^{2\pi} u_0(x,y) e^{-i \overline x\cdot\overline m} dx dy,$$
and $\overline x\cdot\overline m:=xm+yn$.\\

The following lemma is a localized version of the Strichartz estimate satisfied by the group $\{W(t)\}_{t\in\mathbb R}$.
\begin{lemma}\label{L3-1} Let $\{W(t)\}_{t\in\mathbb R}$ be the group defined in \eqref{group} and $\widetilde{P_N}$ the operator defined in the subsection 2.7. Let $\alpha>1/4$. Then for any $u_0\in L^2(\mathbb T^2)$ and any time interval $I\subset \mathbb R$ with $|I|\sim (1 \lor N)^{-1} $ $(N\in P)$,
\begin{align}
\|W(\cdot_t)\widetilde{P_N} u_0\|_{L^2(I;L^\infty (\mathbb T^2))}\lesssim (1\lor N)^\alpha \|\widetilde{P_N}u_0\|_{L^2(\mathbb T^2)}.\label{L3-1-eq3-2}
\end{align}
Moreover, for $s=1/2+\alpha$,
\begin{align}
\|W(\cdot_t)u_0\|_{L^2([0,1];L^\infty(\mathbb T^2))}\lesssim \|u_0\|_{H^s(\mathbb T^2)}.\label{L3-1-eq3-3}
\end{align}
\end{lemma}

In the proof of Lemma \ref{L3-1} we will use the following results.

\begin{lemma}\label{L3-2} (Poisson Summation Formula). (See \cite{G2008}, Theorem 3.1.17, p. 171). Let us suppose that $f$, $\mathcal F_{\mathbb R^n}f\equiv \widehat f$ are in $L^1(\mathbb R^n)$ and satisfy the condition
$$|f(x)|+|\widehat f(x)|\leq C(1+|x|)^{-n-\delta},$$
for some $C,\delta>0$. Then $f$ and $\widehat f$ are continuous and
\begin{align}
\sum_{m\in \mathbb Z^n} \widehat f(2\pi m)=\sum_{m\in\mathbb Z^n}f(m).\label{L3-2-eq3-4}
\end{align}
\end{lemma}

Let us observe that $f\in \mathcal S(\mathbb R^n)$ (Schwartz space) satisfies the condition of this lemma.

\begin{lemma}\label{dirichlet} (Dirichlet). (See \cite{N1996}, Theorem 4.1, p. 98). Let $\alpha$ and $Q$ be real numbers with $Q\geq 1$. Then there exist $a\in\mathbb Z$ and $q\in \mathbb Z^+$, with greatest common divisor equals 1 when $a\neq 0$, such that
\begin{align}
1\leq q\leq Q\quad \text{and}\quad \left|\alpha-\frac aq\right|<\frac1{qQ}.\label{L3-3-eq3-5}
\end{align}
\end{lemma}

\begin{lemma}\label{weyl} (Weyl's inequality). (See \cite{N1996}, Theorem 4.3, p. 114). Let $f(x):=\alpha x^k+\cdots$ be a polynomial with real coefficients, with degree $k$, $k\geq 2$, and suppose that $\alpha$ has the rational approximation $\frac aq$ such that $\displaystyle{\left|\alpha- \frac aq\right|\leq \frac1{q^2}}$, where $q\geq 1$. Let $\displaystyle{S(f):=\sum_{n=1}^N e^{2\pi i f(n)}}$, $K:=2^{k-1}$, and $\epsilon>0$. Then
\begin{align}
|S(f)|\leq C_{\epsilon,k} N^{1+\epsilon}(N^{-1}+q^{-1}+q N^{-k})^{1/K},\label{L3-4-eq3-6}
\end{align}
where the constant $C_{\epsilon,k}$ only depends on $\epsilon$ and $k$.
\end{lemma}

\subsection{Proof of Lemma \ref{L3-1}} We begin with the proof of estimate \eqref{L3-1-eq3-2}. Since for $N=0$ the proof of estimate \eqref{L3-1-eq3-2} is straightforward, we suppose $N=2^k$ for some integer $k\geq0$.
Let $\psi_0:\mathbb R\to[0,1]$ be a smooth even function, with support in $[-2,2]$, such that $\psi_0\equiv1$ in $[-1,1]$. For $\overline m:=(m,n)\in \mathbb Z^2$ define $a(\overline m):=(\widetilde{P_N}w_0)^\wedge(\overline m)$. Since, for $a(\overline m)\neq 0$, it can be seen that
$$\psi_0\left(\frac m{2^k} \right)=\psi_0\left(\frac n{2^k} \right)=1,$$
then
$$[W(t)\widetilde{P_N}u_0](\overline x)=\frac 1\pi \sum_{m\in\mathbb Z-\{0\}} \sum_{n\in\mathbb Z} a(\overline m)\psi_0\left(\frac m{2^k}\right) \psi_0 \left( \frac n{2^k}\right) e^{i[\overline m\cdot \overline x-t\sgn(m)(m^2+n^2)]}.$$
To obtain \eqref{L3-1-eq3-2} it is enough to show that there is $C_\alpha>0$ such that
\begin{align}
\left\|\chi_{[0,N^{-1}]}(|t|) \sum_{m\in\mathbb Z-\{0\}} \sum_{n\in\mathbb Z} a(\overline m)\psi_0\left(\frac m{2^k} \right) \psi_0\left( \frac n{2^k}\right) e^{i[\overline m \cdot (x(t),y(t))-t\sgn(m)(m^2+n^2)]}\right\|_{L^2_t}\leq  C_\alpha N^\alpha \|a\|_{l^2(\mathbb Z^2)},\label{L3-1-eq4}
\end{align}
for any pair of measurable functions $x,y:[-N^{-1},N^{-1}]\to\mathbb T$. Let us denote by $\overline x$  the pair of measurable functions $(x,y)$. Then, by a duality argument, to prove \eqref{L3-1-eq4} is equivalent to prove that for any $g\in L^2_t$, $g:\mathbb R\to\mathbb R$,
\begin{align}
\left| \int_\mathbb R \int_\mathbb R g(t) g(t') K_k(t,t',\overline x)dt dt'\right|\leq C_\alpha N^{2\alpha} \|g\|_{L^2_t}^2,\label{L3-1-eq5}
\end{align}
where
$$K_k(t,t',\overline x):=\chi_{[0,N^{-1}]}(|t|) \chi_{[0,N^{-1}]}(|t'|)\sum_{m\in\mathbb Z-\{0\}} \sum_{n\in\mathbb Z} \psi_0^2\left(\frac m{2^k}\right) \psi_0^2\left(\frac n{2^k}\right) e^{i[\overline m\cdot(\overline x(t)-\overline x(t'))-(t-t')\sgn(m)(m^2+n^2)]}.$$
For $j\geq k$, let us define $K_k^j(t,t',\overline x):=\chi_{(2^{-j},2\cdot 2^{-j}]}(|t-t'|) K_k(t,t',\overline x)$. Then, for $t,t'\in[-N^{-1},N^{-1}]$, $t\neq t'$, it is clear that
$$K_k(t,t',\overline x)=\sum_{j=k}^\infty K_k^j(t,t',\overline x).$$
This way, we obtain
$$\left| \int_\mathbb R\int_\mathbb R g(t) g(t')K_k(t,t',\overline x)dt dt'\right|\leq \sum_{j=k}^\infty \underset{A_{kj}}\iint |g(t)g(t')| dt dt' \sup_{(t,t')\in A_{kj}} |K_k^j (t,t',\overline x)|,$$
where for $j\geq k$
$$A_{kj}:=\{(t,t'):|t|\leq N^{-1},|t'|\leq N^{-1},2^{-j}<|t-t'|\leq 2\cdot 2^{-j}\}.$$

It is not difficult to show that
$$\underset{A_{kj}}\iint |g(t) g(t')| dt dt' \lesssim 2^{-j} \|g\|^2_{L^2_t}.$$
Therefore, we have that
$$\left|\int_{\mathbb R} \int_{\mathbb R}g(t)g(t') K_k(t,t',\overline x) dt dt' \right| \lesssim \left(\sum_{j=k}^\infty 2^{-j} \sup_{(t,t')\in A_{kj}} |K_k^j (t,t',\overline x)|\right)\|g\|_{L^2_t}^2.$$
If we prove that
\begin{align}
\sup_{(t,t')\in A_{kj}} |K_k^j (t,t',\overline x)|\lesssim 2^j 2^{2\alpha k} 2^{-\epsilon j}\label{L3-1-eq6}
\end{align}
for some $\epsilon>0$, then we have \eqref{L3-1-eq5}. Hence, let us prove \eqref{L3-1-eq6}. Estimate \eqref{L3-1-eq6} follows if we show that for any $\overline x=(x,y)\in\mathbb T^2$ and any $t$, such that $|t|\in(2^{-j},2\cdot 2^{-j}]$,
\begin{align}
\Big|\sum_{m\in\mathbb Z-\{0\}} \sum_{n\in\mathbb Z} \psi_0^2\left(\frac m{2^k}\right) e^{i[mx-t\sgn(m)m^2]}\psi_0^2\left(\frac n{2^k}\right) e^{i[ ny -t\sgn(m)n^2]}\Big|\lesssim 2^j 2^{2\alpha k} 2^{-\epsilon j}.\label{L3-1-eq7}
\end{align}
Let us observe that
\begin{align*}
\sum_{m\in\mathbb Z-\{0\}} \sum_{n\in\mathbb Z} \psi_0^2\left(\frac m{2^k}\right) e^{i[mx-t\sgn(m)m^2]} &\psi_0^2\left(\frac n{2^k}\right) e^{i[ ny -t\sgn(m)n^2]}\\
=&\sum_{m=-2N}^{-1} \psi_0^2\left(\frac m{2^k}\right) e^{i[mx+tm^2]}  \left( \sum_{n\in\mathbb Z}  \psi_0^2\left(\frac n{2^k}\right) e^{i[ ny +t n^2]}\right)\\
&+\sum_{m=1}^{2N} \psi_0^2\left(\frac m{2^k}\right) e^{i[mx-t m^2]}  \left( \sum_{n\in\mathbb Z}  \psi_0^2\left(\frac n{2^k}\right) e^{i[ ny -t n^2]}\right).
\end{align*}

Let us bound the sum
\begin{align}
\sum_{m=1}^{2N} \psi_0^2\left(\frac m{2^k}\right) e^{i[mx-t m^2]}  \left( \sum_{n\in\mathbb Z}  \psi_0^2\left(\frac n{2^k}\right) e^{i[ ny -t n^2]}\right),\label{L3-1-eq8}
\end{align}
being the estimate of the other sum similar. Using Lemma \ref{L3-2} in the sum with index $n$, we conclude that \eqref{L3-1-eq8} is the same that
$$\sum_{m=1}^{2N} \psi_0^2\left(\frac m{2^k}\right) e^{i[mx-t m^2]}  \left( \sum_{\nu \in\mathbb Z} \int_\mathbb R \psi_0^2\left(\frac \eta{2^k}\right) e^{i[ (y -2\pi \nu ) \eta-\eta^2 t]}d\eta \right).$$
Let us estimate the integral
$$\int_\mathbb R \psi_0^2\left(\frac \eta{2^k}\right) e^{i[ (y -2\pi \nu ) \eta-\eta^2 t]}d\eta.$$
Using integration by parts we obtain
\begin{align}
\notag\int_\mathbb R \psi_0^2\left( \frac\eta{2^k}\right) e^{i[(y-2\pi \nu)\eta-\eta^2 t]}d\eta&=\int_\mathbb R \frac{\psi^2_0\left( \frac\eta{2^k} \right)}{i[(y-2\pi\nu)-2\eta t]} \frac d{d\eta} e^{i[(y-2\pi \nu)\eta-\eta^2 t]} d\eta\\
&=-\frac 1i \int_{|\eta|\leq 2N} \left[ \frac{2\psi_0 \psi'_0 \cdot \frac1{2^k}}{(y-2\pi\nu)-2\eta t}  -\frac{\psi_0^2\cdot (-2t)} {[(y-2\pi\nu)-2\eta t]^2}   \right]e^{i[(y-2\pi \nu)\eta-\eta^2 t]}d\eta.\label{L3-1-eq9}
\end{align}
Let us observe that
$$|-2t|\leq 2\cdot 2\cdot 2^{-j}=4\cdot 2^{-j}\leq 4\quad \text{and}\quad|2\eta t|\leq 4N|t|\leq 8\cdot 2^k\cdot 2^{-j}\leq 8.$$
Since $|2\eta t|\leq 8$ and $y\in[0,2\pi)$, then, if $|\nu|\geq 100$, $|(y-2\pi\nu)-2\eta t|\geq \pi|\nu|$. Thus, for $|\nu|\geq 100$,
\begin{align}
\left|\int_{|\eta|\leq 2N}\frac{\psi_0^2\cdot (-2t)}{[(y-2\pi\nu)-2\eta t]^2} e^{i[(y-2\pi \nu)\eta-\eta^2 t]} d\eta\right|\leq 4\int_{|\eta|\leq 2N} \frac{2^{-j}}{\pi^2 \nu^2} d\eta\leq\frac{4\cdot 2^{-j}\cdot 4N}{\pi^2\nu^2}\leq\frac{16}{\pi^2 \nu^2}.\label{L3-1-eq10}
\end{align}
On the other hand, using integration by parts
\begin{align*}
\int_{|\eta|\leq 2N} \frac{2\psi_0 \psi'_0 \cdot \frac1{2^k}}{(y-2\pi\nu)-2\eta t} e^{i[(y-2\pi \nu)\eta-\eta^2 t]}d\eta&=\int_{|\eta|\leq 2N} \frac{2\psi_0 \psi'_0\cdot 2^{-k}}{i[(y-2\pi\nu)-2\eta t]^2}\frac{d}{d\eta} e^{i[(y-2\pi \nu)\eta-\eta^2 t]} d\eta\\
&=-\frac 1i \int_{|\eta|\leq 2N} \left[ \frac{2\cdot 2^{-2k}{\psi'_0}^2+2\cdot 2^{-2k} \psi_0 \psi''_0} {[(y-2\pi \nu)-2\eta t]^2}\right.\\
&\hspace{2.7cm}\left.-\frac{4\cdot 2^{-k}\psi_0\psi'_0\cdot(-2t)}{[(y-2\pi \nu)-2\eta t]^3}\right] e^{i[(y-2\pi \nu)\eta-\eta^2 t]} d\eta.
\end{align*}
Therefore, if $|\nu|\geq 100$,
\begin{align}
\Big|\int_{|\eta|\leq 2N} \frac{2\psi_0 \psi'_0 \cdot \frac1{2^k}}{(y-2\pi\nu)-2\eta t} e^{i[(y-2\pi \nu)\eta-\eta^2 t]}d\eta\Big| \leq \frac{2^{-2k} 2^k}{\pi^2 \nu^2}+\frac{C 2^{-k}}{\pi^3\nu^3}\leq \frac C{\pi^2\nu^2}.  \label{L3-1-eq11}
\end{align}

From \eqref{L3-1-eq9}, \eqref{L3-1-eq10} and \eqref{L3-1-eq11}, if $|\nu|\geq 100$, it follows that
$$\left| \int_\mathbb R \psi_0^2\left( \frac\eta{2^k}\right) e^{i[(y-2\pi \nu)\eta-\eta^2 t]}d\eta\right|\leq \frac C{\pi^2\nu^2}.$$
Hence
$$ \sum_{\nu \in\mathbb Z} \int_\mathbb R \psi_0^2\left(\frac \eta{2^k}\right) e^{i[ (y -2\pi \nu ) \eta-\eta^2 t]}d\eta=\sum_{|\nu|\leq 100}  \int_\mathbb R \psi_0^2\left(\frac \eta{2^k}\right) e^{i[ (y -2\pi \nu ) \eta-\eta^2 t]}d\eta+O(1).$$
This way,
\begin{align}
\notag\sum_{m=1}^{2N} \psi_0^2\left(\frac m{2^k}\right) &e^{i[mx-t m^2]}  \left( \sum_{\nu \in\mathbb Z} \int_\mathbb R \psi_0^2\left(\frac \eta{2^k}\right) e^{i[ (y -2\pi \nu ) \eta-\eta^2 t]}d\eta \right)\\
\notag&=\sum_{m=1}^{2N} \psi_0^2\left(\frac m{2^k}\right) e^{i[mx-t m^2]}  \left( \sum_{|\nu|\leq 100} \int_\mathbb R \psi_0^2\left(\frac \eta{2^k}\right) e^{i[ (y -2\pi \nu ) \eta-\eta^2 t]}d\eta +O(1)\right)\\
&=\sum_{m=1}^{2N} \psi_0^2\left(\frac m{2^k}\right) e^{i[mx-t m^2]}  \left( \sum_{|\nu|\leq 100} \int_\mathbb R \psi_0^2\left(\frac \eta{2^k}\right) e^{i[ (y -2\pi \nu ) \eta-\eta^2 t]}d\eta\right)+O(2^k).\label{L3-1-eq12}
\end{align}
Since the sum with index $\nu$ in \eqref{L3-1-eq12} has a finite number of terms, then it is enough to estimate
\begin{align}
\sum_{m=1}^{2N} \psi_0^2\left(\frac m{2^k}\right) e^{i[mx-t m^2]}   \int_\mathbb R \psi_0^2\left(\frac \eta{2^k}\right) e^{i[ y'\eta -\eta^2 t]}d\eta,\quad y'\in\mathbb R.\label{L3-1-eq13}
\end{align}
Let us observe that
\begin{align}
 \int_\mathbb R \psi_0^2\left(\frac \eta{2^k}\right) e^{i[ y'\eta -\eta^2 t]}d\eta=2^k\int_\mathbb R \psi_0^2(\eta) e^{i[y' 2^k \eta-2^{2k}t\eta^2]} d\eta=2^k\left(\mathcal F^{-1}\left(\psi_0^2(\cdot_\eta)e^{-i 2^{2k}t \eta^2}\right) \right)(2^k y'). \label{L3-1-eq14}
\end{align}
We define
$$a_m:=e^{i[mx-tm^2]}\quad \text{and}\quad b_m:=\psi_0^2\left(\frac m{2^k} \right) \int_\mathbb R \psi_0^2 \left(\frac\eta{2^k} \right) e^{i[y'\eta-\eta^2 t]} d\eta.$$
Then, from \eqref{L3-1-eq13}, it follows that
\begin{align}
\notag\left|\sum_{m=1}^{2N} \psi_0^2\left(\frac m{2^k}\right) e^{i[mx-t m^2]}   \int_\mathbb R \psi_0^2\left(\frac \eta{2^k}\right) e^{i[ y'\eta -\eta^2 t]}d\eta\right|&=\left| \sum_{m=1}^{2N} a_m b_m\right|=\left| \sum_{l=1}^{2N}\left(\sum_{m=1}^l a_m \right)(b_l-b_{l+1})\right|\\
&\leq \sum_{l=1}^{2N} \left| \sum_{m=1}^l a_m \right| |b_{l}-b_{l+1}|,\label{L3-1-eq15}
\end{align}
where $b_{2N+1}=0$. Now, let us estimate $|b_l-b_{l+1}|$.
\begin{align*}
|b_{l}-b_{l+1}|&=2^k \left| \mathcal F^{-1} \left(\psi_0^2(\cdot_\eta)  e^{-i 2^{2k} t \eta^2}\right)(2^ky')\right| \left|\psi_0^2\left(\frac l{2^k} \right)-\psi_0^2\left(\frac{l+1}{2^k} \right)\right|\leq C\frac{2^k}{2^k}\left\| \mathcal F^{-1} \left(\psi_0^2(\cdot_\eta) \right) \ast \mathcal F^{-1} (e^{-i 2^{2k}t\eta^2} ) \right\|_{L^\infty_y}\\
&\leq C \left\| \mathcal F^{-1}(\psi_0^2(\cdot_\eta)) \right\|_{L^1} \|\mathcal F^{-1}(e^{-i 2^{2k}t\eta^2})\|_{L^\infty}\leq C\||2^{2k}t|^{-1/2} e^{i\sgn(-2^{2k}t) \pi/4} e^{i(-2^{2k}t)^{-1}y^2}\|_{L^\infty_y}\\
&=\frac C{2^k|t|^{1/2}}\leq \frac{C 2^{j/2}}{2^k}.
\end{align*}
Therefore
\begin{align}
\left|\sum_{m=1}^{2N} \psi_0^2\left(\frac m{2^k}\right) e^{i[mx-t m^2]}   \int_\mathbb R \psi_0^2\left(\frac \eta{2^k}\right) e^{i[ y'\eta -\eta^2 t]}d\eta\right|\leq C\frac{2^{j/2}}{2^k}\sum_{l=1}^{2N} \left| \sum_{m=1}^l a_m \right|.\label{L3-1-eq16}
\end{align}
Let $\widetilde \alpha$ be a number in $(0,1)$. Then, from \eqref{L3-1-eq16},
\begin{align}
\left|\sum_{m=1}^{2N} \psi_0^2\left(\frac m{2^k}\right) e^{i[mx-t m^2]}   \int_\mathbb R \psi_0^2\left(\frac \eta{2^k}\right) e^{i[ y'\eta -\eta^2 t]}d\eta\right|\leq C\frac{2^{j/2}}{2^k} \left( \sum_{1\leq l\leq N^{\widetilde \alpha}} l+\sum_{N^{\widetilde \alpha}<l\leq2N} \left| \sum_{m=1}^l a_m \right| \right).\label{L3-1-eq17}
\end{align}
We will estimate $\displaystyle{\left|\sum_{m=1}^l a_m \right|}$ for $N^{\widetilde\alpha}<l\leq 2N$. Let us observe that
\begin{align}
\left| \sum_{m=1}^l a_m \right|=\left| \sum_{m=1}^l e^{2\pi i f(m)}\right|,\label{L3-1-eq18}
\end{align}
where
$f(m):=-\frac t{2\pi}m^2+\frac x{2\pi}m$.\\

We will consider two cases about the index $j$.\\

(i) Let us assume that $k\leq j\leq \beta k$, with $\beta>1$. Let $Q:=l^\gamma$ (for some $\gamma>0$ to be determined later) and $\alpha:=-\frac t{2\pi}$, where
$$\left|-\frac t{2\pi} \right|\in\left(\frac{2^{-j}}{2\pi},\frac{2^{-j}}{\pi} \right].$$
Form \eqref{L3-3-eq3-5} in Lemma \ref{dirichlet}, we have that there exist $a\in\mathbb Z$ and $q\in \mathbb Z^+$, such that
\begin{align}
1\leq q\leq l^\gamma\quad \text{and}\quad \left|-\frac t{2\pi}-\frac aq\right|<\frac1{ql^\gamma}.\label{L3-1-eq19}
\end{align}
If $a=0$, from \eqref{L3-1-eq19},
$$\left|\frac t{2\pi} \right|\leq \frac1{q l^\gamma}\leq \frac 1{l^\gamma}\leq \frac 1{N^{\widetilde \alpha \gamma}}=\frac1{2^{\widetilde \alpha \gamma k}}.$$
But
$$\left| \frac t{2\pi}\right|>\frac1{2\pi 2^j}\geq \frac 1{2\pi 2^{\beta k}} .$$
Hence
$$\frac1{2\pi 2^{\beta k}}\leq \frac1{2^{\widetilde \alpha\gamma k}},$$
this is $2^{(\widetilde \alpha\gamma-\beta)k}\leq 2\pi$ and this is impossible if $\widetilde \alpha\gamma-\beta>0$. \\

In consequence if we suppose that $\widetilde\alpha\gamma>\beta$, then necessarily $a\neq 0$, and from \eqref{L3-1-eq19} it follows that
$$\left|\frac{qt}{2\pi}+a \right|<\frac1{l^\gamma},$$
and this implies that
$$\frac{q|t|}{2\pi}>|a|-\frac1{l^\gamma}\geq1-\frac1{l^\gamma}\geq1-\frac1{N^{\widetilde\alpha\gamma}}=1-\frac1{2^{\widetilde\alpha \gamma k}}\geq\left(1-\frac1{2^{\widetilde\alpha\gamma}} \right).$$
Hence
$$
q\geq\frac{2\pi}{|t|}\left( 1-\frac1{2^{\widetilde\alpha\gamma}}\right)\geq \pi \left( 1-\frac1{2^{\widetilde\alpha\gamma}}\right)2^j\geq C 2^k,
$$
i.e.
\begin{align}
\frac1q\leq\frac{\widetilde C}{2^k}=\frac{\widetilde C}N.\label{L3-1-eq20}
\end{align}
Since
$$\left|-\frac t{2\pi}-\frac aq\right|\leq \frac1{ql^\gamma}\leq\frac1{q^2},$$
from Lemma \ref{weyl}, with
$$f(z)=-\frac t{2\pi} z^2+\frac x{2\pi} z,$$
we can conclude that, for every $\epsilon>0$, there exists $C_\epsilon>0$, such that
$$\left|\sum_{m=1}^l a_m\right|=\left|\sum_{m=1}^l e^{2\pi i f(m)} \right|\leq C_\epsilon l^{1+\epsilon}\left[\frac1 l+\frac1q+\frac q{l^2} \right]^{1/2}.$$
Taking into account \eqref{L3-1-eq20} and the fact that $q\leq l^\gamma$, it follows that
$$\left|\sum_{m=1} ^l a_m\right|\leq C_\epsilon l^{1+\epsilon}\left[\frac1l+\frac1N+\frac{l^\gamma}{l^2} \right]^{1/2}\leq C_\epsilon l^{1+\epsilon}\left[\frac1l+\frac1{l^{2-\gamma}} \right]^{1/2}. $$
In this manner, in order to have a non trivial estimate, it is required that $\gamma<2$. Thus,  our parameters $\widetilde\alpha$, $\beta$, and $\gamma$ must meet the conditions
$$0<\widetilde\alpha<1<\beta<\widetilde\alpha\gamma<\gamma<2.$$
Since $2-\gamma<1$, for $N^{\widetilde\alpha}<l\leq 2N$,
$$\left|\sum_{m=1} ^l a_m\right|\leq C_\epsilon l^{1+\epsilon}\left[\frac1l+\frac1{l^{2-\gamma}} \right]^{1/2}=C_\epsilon l^{\frac\gamma2+\epsilon}.$$
In consequence, from \eqref{L3-1-eq17}, it follows that
\begin{align}
\notag\left|\sum_{m=1}^{2N}\psi_0^2\left(\frac m{2^k}\right)e^{i[mx-tm^2]}\int_{\mathbb R}\psi_0^2\left(\frac\eta{2^k} \right) e^{i[y'\eta-\eta^2t]}d\eta\right|&\leq C\frac{2^{\frac{j}2}}{2^k}\left(N^{2\widetilde\alpha}+\sum_{N^{\widetilde\alpha}<l\leq 2N}C_\epsilon l^{\frac\gamma2+\epsilon}\right)\\
\notag &\leq C_\epsilon \frac{2^{\frac{j}2}}{2^k}\left[N^{2\widetilde\alpha} +N^{1+\frac\gamma2+\epsilon}\right]\leq C_\epsilon 2^{j/2}\left[2^{(2\widetilde\alpha-1)k}+2^{(\frac\gamma2+\epsilon)k} \right]\\
&\leq C_\epsilon 2^{j} \left[2^{(2\widetilde \alpha-\frac32+\epsilon)k}+2^{(\frac{\gamma-1}2+2\epsilon)k} \right]2^{-\epsilon j},\label{L3-1-eq21}
\end{align}
for $0<\epsilon<1/2$.\\ 

(ii) Let us assume that $j>\beta k$, then for $0<\epsilon<1/2$,
\begin{align}
\notag\left|\sum_{m=1}^{2N}\psi_0^2\left(\frac m{2^k}\right)e^{i[mx-tm^2]}\int_{\mathbb R}\psi_0^2\left(\frac\eta{2^k} \right) e^{i[y'\eta-\eta^2t]}d\eta\right| &\leq C\frac{2^{j/2}}{2^k}\sum_{l=1}^{2N} \left|\sum_{m=1}^l a_m \right|\leq C\frac{2^{j/2}}{2^k}\sum_{l=1}^{2N} l\leq C\frac{2^{j/2}}{2^k} N^2\\
&= C 2^{j/2}2^k =C 2^j 2^k 2^{-(\frac12-\epsilon+\epsilon)j}\leq C2^j 2^{(1-(\frac12-\epsilon)\beta)k}2^{-\epsilon j}.\label{L3-1-eq22}
\end{align}
If $\widetilde \alpha:=1-\epsilon/2$, $\gamma:=2-4\epsilon$, $\beta:=1+\epsilon$, then
$$\widetilde\alpha\gamma>\beta\quad \text{if and only if}\quad 0<\epsilon<\frac{6-\sqrt{28}}4.$$
Besides
$$\left(2\widetilde\alpha-\frac32+\epsilon \right)=\frac12,\quad \frac{\gamma-1}2+2\epsilon=\frac12,\quad 1-\left(\frac12-\epsilon\right)\beta=\frac12+\frac\epsilon2+\epsilon^2.$$
Hence, the greatest exponent in \eqref{L3-1-eq21} and \eqref{L3-1-eq22} which contains $k$ is
$$1-\left(\frac12-\epsilon \right)\beta.$$
Therefore, for any $j\geq k$, if $|t|\in(2^{-j},2\cdot 2^{-j}]$, it is true that
$$\left|\sum_{m=1}^{2N}\psi_0^2\left(\frac m{2^k}\right)e^{i[mx-tm^2]}\int_{\mathbb R}\psi_0^2\left(\frac\eta{2^k} \right) e^{i[y'\eta-\eta^2t]}d\eta\right|\leq C_\epsilon 2^j 2^{(\frac12+\frac\epsilon2+\epsilon^2)k}2^{-\epsilon j},$$
for $0<\epsilon<\frac{6-\sqrt{28}}4$. Let us observe that
$$2^k=2^j2^k2^{-j}=2^j2^k2^{-(1-\epsilon+\epsilon)j}\leq 2^j2^{(1-(1-\epsilon))k}2^{-\epsilon j}=2^j2^{\epsilon k}2^{-\epsilon j},$$
then, from \eqref{L3-1-eq12} we have that
$$\left|\sum_{m=1}^{2N} \psi_0^2\left(\frac m{2^k}\right) e^{i[mx-t m^2]}  \left( \sum_{\nu \in\mathbb Z} \int_\mathbb R \psi_0^2\left(\frac \eta{2^k}\right) e^{i[ (y -2\pi \nu ) \eta-\eta^2 t]}d\eta \right)\right|\leq C_\epsilon 2^j 2^{(\frac12+\epsilon)k}2^{-\epsilon j},$$
if $0<\epsilon<\frac{6-\sqrt{28}}4$. i.e.,
$$\left|\sum_{m=1}^{2N} \psi_0^2\left(\frac m{2^k}\right) e^{i[mx-t m^2]}  \left( \sum_{n \in\mathbb Z} \int_\mathbb R \psi_0^2\left(\frac n{2^k}\right) e^{i[ ny -tn^2] } \right)\right|\leq C_\epsilon 2^j 2^{(\frac12+\epsilon)k}2^{-\epsilon j},$$
for $0<\epsilon<\frac{6-\sqrt{28}}4$, if $(x,y)\in \mathbb T^2$ and $|t|\in(2^{-j},2\cdot 2^{-j}]$. This way, if $2\alpha>1/2$. i.e., $\alpha>1/4$, we obtain the estimate \eqref{L3-1-eq7}. And as it was stated previously, this implies \eqref{L3-1-eq3-2}.\\

Let us prove now \eqref{L3-1-eq3-3}. We divide $[0,1]$ into $N$ intervals of lenght $N^{-1}$. Using \eqref{L3-1-eq3-2} it follows that
\begin{align*}
\|W(\cdot_t)\widetilde{P_N} u_0\|^2_{L^2([0,1];L^\infty_{xy})}&=\sum_{i=1}^N \int_{(\frac{i-1}N,\frac iN]} \|W(t)\widetilde{P_N}u_0\|^2_{L^\infty_{xy}}dt\\
&\lesssim C_\alpha^2 \sum_{i=1}^N N^{2\alpha}  \|\widetilde{P_N}u_0\|_{L^2_{xy}}^2\leq C_\alpha^2 N N^{2\alpha} \|u_0\|^2_{L^2_{xy}},
\end{align*}
which leads to
\begin{align}
\|W(\cdot_t)\widetilde{P_N} u_0\|_{L^2([0,1];L^\infty_{xy})}\leq C_\alpha N^{\frac12+\alpha} \|u_0\|_{L^2_{xy}}.\label{L3-1-eq25}
\end{align}
From \eqref{L3-1-eq25} and the fact that $\widetilde{P_N}^2=\widetilde{P_N}$, it follows that
\begin{align*}
\|W(\cdot_t)&u_0\|_{L^2([0,1];L^\infty_{xy})}\\
&\lesssim\sum_{N\in P} \|W(\cdot_t) \widetilde{P_N} u_0\|_{L^2([0,1];L^\infty_{xy})}=\|W(\cdot_t)\widetilde{P_0} u_0\|_{L^2([0,1];L^\infty_{xy})}+\sum_{N\in\{1,2,2^2,\dots\}} \|W(\cdot_t)\widetilde{P_N} u_0\|_{L^2([0,1];L^\infty_{xy})}\\
&\lesssim\|\widetilde{P_0} u_0\|_{L^2_{xy}}+\sum_{N\in\{1,2,2^2,\dots\}} C_\alpha N^{1/2+\alpha}\|\widetilde{P_N} u_0\|_{L^2_{xy}}\lesssim\|u_0\|_{L^2_xy}+C_\alpha \sum_{N\in\{1,2,2^2,\dots\}} \|\widetilde{P_N} u_0\|_{H^{\frac12+\alpha}(\mathbb T^2)}\\
&\leq C_\alpha \|u_0\|_{H^{\frac12+\alpha}(\mathbb T^2)}.
\end{align*}
(The last inequality follows since $N\neq N'$ implies that $\supp(\widetilde{P_N} u_0)^\wedge\cap\supp(\widetilde{P_{N'}}u_0)^\wedge=\phi$).\qed\\

As a consequence of Lemma \ref{L3-1} we have the following estimate for solutions of a non homogeneous linear problem.

\begin{corollary}\label{C3-5} Suppose that $w$ is a smooth solution of the non homogeneous linear problem
\begin{align}
w_t+\mathcal H \Delta w+\partial_x F(w)=0,\quad (x,y)\in\mathbb T^2,\quad t\in[0,T],\label{3.6}
\end{align}
where $F(w)\in L^1([0,T]; H^s(\mathbb T^2))$, with $s>3/4$. Then
\begin{align}
\|w\|_{L^1_T L^\infty(\mathbb T^2)}\lesssim T^{1/2}(\|w\|_{L^\infty_T H^s(\mathbb T^2)}+\|F(w)\|_{L^1_T H^s(\mathbb T^2)}). \label{L3-5eq3-7}
\end{align}
\end{corollary}
\begin{proof}
Let $N:=2^k$ for some $k\in\mathbb N\cup\{0\}$. Let us split the interval $[0,T]$ in subintervals $[a_i,a_{i+1})$ of length $\frac{T}{N}$ for $i=1,\cdots,N$. Then
\begin{align}
\|\widetilde{P_N}w\|_{L^1_TL^{\infty}(\mathbb T^2)}&\leq\sum_{i=1}^N\|\chi_{[a_i,a_{i+1})}(\cdot_t)\widetilde{P_N}w\|_{L^1_TL^{\infty}(\mathbb T^2)}\leq\sum_{i=1}^N\big(\int_{a_i}^{a_{i+1}}dt\big)^{\frac12}\|\chi_{[a_i,a_{i+1})}(\cdot_t)\widetilde{P_N}w\|_{L^2_TL^{\infty}(\mathbb T^2)}\notag\\
&\lesssim\frac{T^{\frac12}}{N^{\frac12}}\sum_{i=1}^N\|\chi_{[a_i,a_{i+1})}(\cdot_t)\widetilde{P_N}w\|_{L^2_TL^{\infty}(\mathbb T^2)}\,.\label{3.8}
\end{align}
Using the Duhamel's formula, from \eqref{3.6} we obtain for $t\in[a_i,a_{i+1})$:
\begin{align}
\chi_{[a_i,a_{i+1})}(t)\widetilde{P_N}w(t)=W(t-a_i)\widetilde{P_N}w(a_i)-\int_{a_i}^tW(t-s)\widetilde{P_N}[\partial_xF(w)(s)]ds .\notag
\end{align}
Taking into account \eqref{L3-1-eq3-2} from Lemma \ref{L3-1}, it follows from the latter equality that for $\alpha:=s-\frac12>\frac14$
\begin{align}
\|\chi_{[a_i,a_{i+1})}(\cdot_t)\widetilde{P_N}w\|_{L^2_TL^{\infty}(\mathbb T^2)}&\lesssim N^{\alpha}\|\widetilde{P_N}w(a_i)\|_{L^2(\mathbb T^2)}+\int_{a_i}^tN^{\alpha}\|\widetilde{P_N}\partial_xF(w)(s)\|_{L^2(\mathbb T^2)}ds\notag\\
&\lesssim N^{\alpha}\|\widetilde{P_N}w(a_i)\|_{L^2(\mathbb T^2)}+N^{\alpha}N\int_{a_i}^t\|\widetilde{P_N}(F(w)(s))\|_{L^2(\mathbb T^2)}ds\notag\\
&\lesssim N^{\alpha}\|\widetilde{P_N}w(a_i)\|_{L^2(\mathbb T^2)}+N^{\alpha}N\|\chi_{[a_i,a_{i+1})}(\cdot_t)\widetilde{P_N}F(w)\|_{L^1_TL^2(\mathbb T^2)}\,.\label{3.10}
\end{align}
Combining \eqref{3.8} and \eqref{3.10} we have that
\begin{align}
\|\widetilde{P_N}w\|_{L^1_TL^{\infty}(\mathbb T^2)}&\lesssim\frac{T^{\frac12}}{N^{\frac12}}\sum_{i=1}^N\big(N^{\alpha}\|\widetilde{P_N}w(a_i)\|_{L^2(\mathbb T^2)}+N^{\alpha}N\|\chi_{[a_i,a_{i+1})}(\cdot_t)\widetilde{P_N}F(w)\|_{L^1_TL^2(\mathbb T^2)}\big)\notag\\
&\lesssim\frac{T^{\frac12}}{N^{\frac12-\alpha}}\sum_{i=1}^N\|\widetilde{P_N}w(a_i)\|_{L^2(\mathbb T^2)}+T^{\frac12}N^{\alpha}N^{\frac12}\sum_{i=1}^N\|\chi_{[a_i,a_{i+1})}(\cdot_t)\widetilde{P_N}F(w)\|_{L^1_TL^2(\mathbb T^2)}\notag\\
&\lesssim\frac{T^{\frac12}}{N^{\frac12-\alpha}N^s}\sum_{i=1}^NN^s\|\widetilde{P_N}w(a_i)\|_{L^2(\mathbb T^2)}+T^{\frac12}N^s\|\widetilde{P_N}F(w)\|_{L^1_TL^2(\mathbb T^2)}\notag\\
&\lesssim\frac{T^{\frac12}}{N}\sum_{i=1}^N\|\widetilde{P_N}w(a_i)\|_{H^s(\mathbb T^2)}+T^{\frac12}\|\widetilde{P_N}F(w)\|_{L^1_TH^s(\mathbb T^2)}\notag\\
&\lesssim T^{\frac12}\|\widetilde{P_N}w\|_{L^{\infty}_TH^s(\mathbb T^2)}+T^{\frac12}\|\widetilde{P_N}F(w)\|_{L^1_TH^s(\mathbb T^2)}.\label{3.12}
\end{align}
Using \eqref{3.12} we have that
\begin{align}
\|w\|_{L^1_TL^{\infty}(\mathbb T^2)}&\leq\|\widetilde{P_0}w\|_{L^1_TL^{\infty}(\mathbb T^2)}+\sum_{N\in\{2^0,2^1,\cdots\}}\|\widetilde{P_N}w\|_{L^1_TL^{\infty}(\mathbb T^2)}\notag\\
&\leq T^{\frac12}\|\widetilde{P_0}w\|_{L^2_TL^{\infty}(\mathbb T^2)}+T^{\frac12}\sum_{N\in\{2^0,2^1,\cdots\}}\big(\|\widetilde{P_N}w\|_{L^{\infty}_TH^s(\mathbb T^2)}+\|\widetilde{P_N}F(w)\|_{L^1_TH^s(\mathbb T^2)}\big).\label{3.13}
\end{align}
But
$$
\widetilde{P_0}w(t)=W(t)\widetilde{P_0}w(0)-\int_{0}^tW(t-s)\widetilde{P_0}[\partial_xF(w)(s)]ds,
$$
and therefore, taking into account \eqref{L3-1-eq3-2} from Lemma \ref{L3-1}, it follows that
\begin{align}
\|\widetilde{P_0}w\|_{L^2_TL^{\infty}(\mathbb T^2)}&\leq \|\widetilde{P_0}w(0)\|_{L^2(\mathbb T^2)}+\int_{0}^t\|\widetilde{P_0}(F(w)(s))\|_{L^2(\mathbb T^2)}ds\notag\\
&\leq \|\widetilde{P_0}w(0)\|_{L^2(\mathbb T^2)}+\|F(w)\|_{L^1_TL^2(\mathbb T^2)}.\label{3.14}
\end{align}
Hence, from \eqref{3.13} and \eqref{3.14}, we can conclude that
\begin{align*}
\|w\|_{L^1_TL^{\infty}(\mathbb T^2)}&\lesssim T^{\frac12}(\|\widetilde{P_0}w(0)\|_{L^2(\mathbb T^2)}+\|F(w)\|_{L^1_TL^2(\mathbb T^2)})+T^{\frac12}(\|w\|_{L^{\infty}_TH^s(\mathbb T^2)}+\|F(w)\|_{L^1_TH^s(\mathbb T^2)})\\
&\lesssim T^{\frac12}(\|w\|_{L^{\infty}_TH^s(\mathbb T^2)}+\|F(w)\|_{L^1_TH^s(\mathbb T^2)}),
\end{align*}
which proves Corollary \ref{C3-5}.
\end{proof}

\section{A Priori Estimates}

In this section we use a Commutator Lemma (see Lemma \ref{DC}) and a Product Lemma (see Lemma \ref{Pro}), in the context of periodic Sobolev spaces, to establish two a priori estimates (see Lemmas \ref{EE} and \ref{gT}) for sufficiently smooth solutions of the two-dimensional BO equation.

\begin{lemma}\label{DC} (A commutator estimate (see \cite{LPRT})) Let $s\geq1$ and $f,g\in H^{\infty}(\mathbb T^2)$. Then 
\begin{align}\|[J^s_{\mathbb T^2},f]g\|_{L^2(\mathbb T^2)}\leq C\big\{\|J^s_{\mathbb T^2}f\|_{L^2(\mathbb T^2)}\|g\|_{L^{\infty}(\mathbb T^2)}+(\|f\|_{L^{\infty}(\mathbb T^2)}+\|\nabla f\|_{L^{\infty}(\mathbb T^2)})\|J^{s-1}_{\mathbb T^2}g\|_{L^2(\mathbb T^2)}\big\}\,,\label{conm}\end{align}
where $[J^s_{\mathbb T^2},f]g:=J^s_{\mathbb T^2}(fg)-fJ^s_{\mathbb T^2}g$.
\end{lemma}
\begin{lemma}\label{Pro}(Product Lemma) Let $s\geq0$ and $f,g\in H^{\infty}(\mathbb T^2)$. Then
\begin{align}
\|fg\|_{H^s(\mathbb T^2)}\lesssim\|f\|_{H^s(\mathbb T^2)}\|g\|_{L^{\infty}(\mathbb T^2)}+\|f\|_{L^{\infty}(\mathbb T^2)}\|g\|_{H^s(\mathbb T^2)}\,.\label{despro}
\end{align}
\end{lemma}
\begin{proof}
The proof of this lemma uses the fact that for $f,g\in H^{\infty}(\mathbb R^2)$ we have
\begin{align}
\|fg\|_{H^s(\mathbb R^2)}\lesssim\|f\|_{H^s(\mathbb R^2)}\|g\|_{L^{\infty}(\mathbb R^2)}+\|f\|_{L^{\infty}(\mathbb R^2)}\|g\|_{H^s(\mathbb R^2)}\,,\label{desproR2}
\end{align}
(see \cite{Tao} pg. 338), and follows the same ideas contained in the proof of Lemma 9.A.1 in \cite{IK20071}.

For $j\in\mathbb Z$ let $A_j$ and $B_j$ be the real intervals $[\frac{\pi}4+2\pi j,\frac{7\pi}4+2\pi j]$ and $[-\frac{3\pi}4+2\pi j, \frac{3\pi}4+2\pi j]$, respectively. Let us consider a partition of unity of $\mathbb R^2$, $\{\phi_{k,l}\}_{(k,l)\in\{1,2\}^2}$, where $\phi_{k,l}:\mathbb R^2\to[0,1]$ is a $2\pi$-periodic function in each variable, $\phi_{k,l}\in C^{\infty}(\mathbb R^2)$ for each $(k,l)$ in $\{1,2\}^2$ and $\phi_{k,l}$ is supported in the set $C_{k,l}$ defined by
\begin{align}
C_{1,1}:=\bigcup_{(j,j')\in\mathbb Z^2}(A_j\times A_{j'}),\qquad C_{1,2}:=\bigcup_{(j,j')\in\mathbb Z^2}(A_j\times B_{j'}),\notag\\
C_{2,1}:=\bigcup_{(j,j')\in\mathbb Z^2}(B_j\times A_{j'}),\qquad C_{2,2}:=\bigcup_{(j,j')\in\mathbb Z^2}(B_j\times B_{j'})\,.\notag
\end{align}
Given $f,g\in H^{\infty}(\mathbb T^2)$ let $\widetilde f$ and $\widetilde g$ in $C^{\infty}(\mathbb R^2)$ be the periodic extensions of $f$ and $g$, respectively. For $(k,l)\in\{1,2\}^2$ let us define the function in $C_0^{\infty}(\mathbb R^2)$:
\[\widetilde{g_{k,l}}:=\widetilde g\phi_{k,l}\chi_{D_{k,l}}\,,\]
where $D_{1,1}:=[0,2\pi)\times[0,2\pi)$, $D_{1,2}:=[0,2\pi)\times[-\pi,\pi)$, $D_{2,1}:=[-\pi,\pi)\times[0,2\pi)$, and $D_{2,2}:=[-\pi,\pi)\times[-\pi,\pi)$ and $\chi_{D_{k,l}}$ is the characteristic function of the set $D_{k,l}$.

Let us denote by $\mathcal{F}_{\mathbb T^2}(fg)(m,n)$ the $(m,n)$-Fourier coefficient of the function $fg$ defined in $\mathbb T^2$ and by $\mathcal{F}_{\mathbb R^2}(\widetilde {f}\widetilde{g_{k,l}})$ the Fourier transform of the function $\widetilde {f}\widetilde{g_{k,l}}$ in $C_0^{\infty}(\mathbb R^2)$. Then
\begin{align} \sum_{(k,l)\in\{1,2\}^2}\mathcal{F}_{\mathbb R^2}(\widetilde f\widetilde{g_{k,l}})(m,n)=\mathcal{F}_{\mathbb T^2}(fg)(m,n)\,.\label{FCP}
\end{align}
In fact, using the periodicity of the functions
 $\widetilde{f}\widetilde{g}\phi_{k,l}$ we have that
\begin{align}
\sum_{(k,l)\in\{1,2\}^2}\mathcal{F}_{\mathbb R^2}(\widetilde f\widetilde{g_{k,l}})(m,n)
&=C\iint_{[0,2\pi)\times[0,2\pi)}e^{-imx-iny}f(x,y)g(x,y)\big(\sum_{(k,l)\in\{1,2\}^2}\phi_{k,l}(x,y)\big)dy\,dx\notag\\
&=C\iint_{[0,2\pi)\times[0,2\pi)}e^{-imx-iny}f(x,y)g(x,y)dy\,dx=\mathcal{F}_{\mathbb T^2}(fg)(m,n)\,,\notag
\end{align}
which proves equality \eqref{FCP}.

Taking into account \eqref{FCP}, for $(x,y)\in\mathbb T^2$ it follows that
\begin{align}
J_{\mathbb T^2}^s(fg)(x,y)&=C\sum_{(m,n)\in\mathbb Z^2}\mathcal{F}_{\mathbb T^2}(fg)(m,n)(1+m^2+n^2)^{\frac{s}2} e^{imx+iny}\notag\\
&=C\sum_{(m,n)\in\mathbb Z^2}\big(\sum_{(k,l)\in\{1,2\}^2}\mathcal{F}_{\mathbb R^2}(\widetilde f\widetilde{g_{k,l}})(m,n)\big)(1+m^2+n^2)^{\frac{s}2} e^{imx+iny}\notag\\
&=C\sum_{(k,l)\in\{1,2\}^2}\big(\sum_{(m,n)\in\mathbb Z^2}\mathcal{F}_{\mathbb R^2}(\widetilde f\widetilde{g_{k,l}})(m,n)(1+m^2+n^2)^{\frac{s}2} e^{imx+iny}\big)\,.\label{proPC}
\end{align}
Using the Poisson summation formula:
\[\sum_{(m,n)\in\mathbb Z^2}F(m,n)=\sum_{(\nu,\lambda)\in\mathbb Z^2}\widehat{F}(2\pi\nu,2\pi\lambda)\]
for
\[F(m,n):=\mathcal{F}_{\mathbb R^2}(\widetilde f\widetilde{g_{k,l}})(m,n)(1+m^2+n^2)^{\frac{s}2} e^{imx+iny}\in\mathcal{S}(\mathbb R^2_{m,n})\,,\]
and bearing in mind that
\begin{align}
\widehat{F}(2\pi\nu,2\pi\lambda)&=C\int_{\mathbb R^2}e^{-i2\pi\nu x'-i2\pi\lambda y'}F(x',y')dx'\,dy'\notag\\
&=C\int_{\mathbb R^2}e^{-i2\pi\nu x'-i2\pi\lambda y'}\mathcal{F}_{\mathbb R^2}(\widetilde f\widetilde{g_{k,l}})(x',y')(1+x'^2+y'^2)^{\frac{s}2} e^{ix'x+iy'y}dx'\,dy'\notag\\
&=C\int_{\mathbb R^2}e^{i(x-2\pi\nu)x'+i(y-2\pi\lambda)y'}\mathcal{F}_{\mathbb R^2}(\widetilde f\widetilde{g_{k,l}})(x',y')(1+x'^2+y'^2)^{\frac{s}2} dx'\,dy'\notag\\
&=J_{\mathbb R^2}^s(\widetilde f\widetilde{g_{k,l}})(x-2\pi\nu,y-2\pi\lambda)\,,\label{DT}
\end{align}
from \eqref{proPC} and \eqref{DT} it follows that for $(x,y)\in\mathbb T^2$
\[J_{\mathbb T^2}^s(fg)(x,y)=C\sum_{(k,l)\in\{1,2\}^2}\big(\sum_{(\nu,\lambda)\in\mathbb Z^2}J^s_{\mathbb R^2}(\widetilde f\widetilde{g_{k,l}})(x-2\pi\nu,y-2\pi\lambda)\big)\,.\]
Hence,
\begin{align}
\|fg\|_{H^s(\mathbb T^2)}&=\|J^s_{\mathbb T^2}(fg)\|_{L^2(\mathbb T^2)}\leq C\sum_{(k,l)\in\{1,2\}^2}\|\sum_{(\nu,\lambda)\in\mathbb Z^2}J^s_{\mathbb R^2}(\widetilde f\widetilde{g_{k,l}})(\cdot_x-2\pi\nu,\cdot_y-2\pi\lambda)\|_{L^2(\mathbb T^2)}\notag\\
&=C\sum_{(k,l)\in\{1,2\}^2}\big(\sum_{(\nu,\lambda)\in\mathbb Z^2}\int_0^{2\pi}\int_0^{2\pi}|J^s_{\mathbb R^2}(\widetilde f\widetilde{g_{k,l}})(x-2\pi\nu,y-2\pi\lambda)|^2dy\,dx\big)^{\frac12}.\label{NPRO}
\end{align}
For $(\nu,\lambda)\in\mathbb Z^2$ let us define
\[E(\nu,\lambda):=\{(x,y)\mid x=x'-2\pi\nu, y=y'-2\pi\lambda\;\text{for some} \;(x',y')\in[0,2\pi)\times[0,2\pi)\}\,.\]
Then, from \eqref{NPRO}, it follows that
\begin{align}
\|fg\|_{H^s(\mathbb T^2)}&\leq C\sum_{(k,l)\in\{1,2\}^2}\big(\sum_{(\nu,\lambda)\in\mathbb Z^2}\|J^s_{\mathbb R^2}(\widetilde f\widetilde{g_{k,l}})\chi_{E(\nu,\lambda)}\|^2_{L^2(\mathbb R^2)}\big)^{\frac12}\notag\\
&=C\sum_{(k,l)\in\{1,2\}^2}\big(\iint_{\mathbb R^2}|J^s_{\mathbb R^2}(\widetilde f\widetilde{g_{k,l}})(x,y)|^2dx\,dy\big)^{\frac12}\notag\\
&=C\sum_{(k,l)\in\{1,2\}^2}\|J^s_{\mathbb R^2}(\widetilde f\widetilde{g_{k,l}})\|_{L^2(\mathbb R^2)}\,.\label{CONT}
\end{align}
Now, let us consider the four intervals $I_1:=[\frac{\pi}{10},\frac{19\pi}{10}]$, $I_2:=[-\frac{9\pi}{10},\frac{9\pi}{10}]$, $J_1:=[\frac{\pi}5,\frac{9\pi}5]$, and $J_2:=[-\frac{4\pi}5,\frac{4\pi}5]$ and let us fix four smooth functions $\widetilde{\Psi_{k,l}}$, with $(k,l)\in\{1,2\}^2$, such that $\widetilde{\Psi_{k,l}}:\mathbb R^2\to[0,1]$ is supported in $I_k\times I_l$ and $\widetilde{\Psi_{k,l}}\equiv 1$ in $J_k\times J_l$. Then $\widetilde f\widetilde{g_{k,l}}=(\widetilde f\,\widetilde{\Psi_{k,l}})\widetilde{g_{k,l}}$, where $\widetilde f\,\widetilde{\Psi_{k,l}}\in H^{\infty}(\mathbb R^2)$ and $\widetilde{g_{k,l}}\in H^{\infty}(\mathbb R^2)$. Therefore, the right hand side of \eqref{CONT} is equal to
\[C\sum_{(k,l)\in\{1,2\}^2}\|J^s_{\mathbb R^2}((\widetilde f\,\widetilde{\Psi_{k,l}})\widetilde{g_{k,l}})\|_{L^2(\mathbb R^2)}\,.\]
Let us point out that from \eqref{desproR2} we have that
\begin{align}
\|J^s_{\mathbb R^2}((\widetilde f\,\widetilde{\Psi_{k,l}})\widetilde{g_{k,l}})\|_{L^2(\mathbb R^2)}&=\|(\widetilde f\,\widetilde{\Psi_{k,l}})\widetilde{g_{k,l}}\|_{H^s(\mathbb R^2)}\notag\\
&\lesssim\|\widetilde f\,\widetilde{\Psi_{k,l}}\|_{H^s(\mathbb R^2)}\|\widetilde{g_{k,l}}\|_{L^{\infty}(\mathbb R^2)}+\|\widetilde f\,\widetilde{\Psi_{k,l}}\|_{L^{\infty}(\mathbb R^2)}\|\widetilde{g_{k,l}}\|_{H^s(\mathbb R^2)}\,.\label{R2}
\end{align}
Let us observe that
\begin{align}
\|\widetilde{g_{k,l}}\|_{L^{\infty}(\mathbb R^2)}\leq \|g\|_{L^{\infty}(\mathbb T^2)}\quad\text{and}\quad\|\widetilde f\,\widetilde{\Psi_{k,l}}\|_{L^{\infty}(\mathbb R^2)}\leq \|f\|_{L^{\infty}(\mathbb T^2)}\,.\label{TRIV}
\end{align}
On the other hand, it can be proved that
\begin{align}
\|\widetilde f\,\widetilde{\Psi_{k,l}}\|_{H^s(\mathbb R^2)}\leq C_s\|f\|_{H^s(\mathbb T^2)}\quad\text{and}\quad\|\widetilde{g_{k,l}}\|_{H^s(\mathbb R^2)}\leq C_s\|g\|_{H^s(\mathbb T^2)}\,.\label{inter}
\end{align}
Then  from \eqref{CONT}, \eqref{R2}, \eqref{TRIV}, and \eqref{inter} it follows the result of Lemma.
\end{proof}
\begin{lemma}\label{EE} Let $s\geq1$ and $T>0$. Let $w\in C([0,T];H^{\infty}(\mathbb T^2))$ be a real solution of the IVP
\begin{align}
\left. \begin{array}{rl}
w_t+ \mathcal H \Delta w +ww_x &\hspace{-2mm}=0,\qquad\qquad (x,y)\in\mathbb T^2,\; t\in[0,T],\\
w(x,y,0)&\hspace{-2mm}=w_0(x,y)\,.
\end{array} \right\}\label{BOS}
\end{align}
Then there exists a positive constant $C_0$ such that
\begin{align}
\|w\|^2_{L^{\infty}_TH^s(\mathbb T^2)}\leq\|w_0\|^2_{H^s(\mathbb T^2)}+C_0(\|w\|_{L^1_TL^{\infty}(\mathbb T^2)}+\|\nabla w\|_{L^1_TL^{\infty}(\mathbb T^2)})\|w\|^2_{L^{\infty}_TH^s(\mathbb T^2)}\,.\label{ENER}
\end{align}
\end{lemma}
\begin{proof} First of all, let us observe that the operator  $\mathcal H\Delta$ is skew-adjoint in $L^2(\mathbb T^2)$. In fact, if we denote by $(\cdot,\cdot)$ the inner product in $L^2(\mathbb T^2)$, it is easy to see that
\begin{align*}
(\mathcal H\Delta u,v)=-(u,\mathcal H \Delta v)\,.
\end{align*}
If we take in the last equality $u=v=w$ a real function, then we obtain
\begin{align}
(\mathcal H\Delta w,w)=0\,.\label{SA}
\end{align}
Applying the operator $J^s_{\mathbb T^2}$ to the equation in \eqref{BOS}, multiplying by $J_{\mathbb T^2}^sw(t)$, integrating in $\mathbb T^2$ and, taking into account \eqref{SA}, we obtain, for $t\in[0,T]$, that
\begin{align}
\frac12\frac d{dt}\|J^s_{\mathbb T^2} w(t)\|^2_{L^2(\mathbb T^2)}+\int_{\mathbb T^2}J^s_{\mathbb T^2}(w(t)\partial_x w(t))J^s_{\mathbb T^2}w(t) dxdy=0.\label{ag02_2}
\end{align}
Using the notation of commutator, integration by parts and Cauchy-Schwarz inequality, from \eqref{ag02_2} we obtain that
\begin{align}
\frac12\frac d{dt}\|J^s_{\mathbb T^2} w(t)\|^2_{L^2(\mathbb T^2)}&=-\int_{\mathbb T^2}([J^s_{\mathbb T^2}, w(t)]\partial_xw(t))J^s_{\mathbb T^2}w(t)dx dy-\int_{\mathbb T^2}w(t)J^s_{\mathbb T^2}(\partial_xw(t))J^s_{\mathbb T^2}w(t)dx\,dy\notag\\
&=-\int_{\mathbb T^2}([J^s_{\mathbb T^2}, w(t)]\partial_xw(t))J^s_{\mathbb T^2}w(t)dx dy+\frac12\int_{\mathbb T^2}\partial_xw(t)(J^s_{\mathbb T^2}w(t))^2dx\,dy\notag\\
&\leq\|[J^s_{\mathbb T^2}, w(t)]\partial_xw(t)\|_{L^2(\mathbb T^2)}\|J^s_{\mathbb T^2}w(t)\|_{L^2(\mathbb T^2)}+\frac12\|\partial_xw(t)\|_{L^{\infty}(\mathbb T^2)}\|J^s_{\mathbb T^2}w(t)\|_{L^2(\mathbb T^2)}^2\,.\label{ag02_3}
\end{align}
In accordance with \eqref{conm} in Lemma \ref{DC} we have
\begin{align}
 \|[J^s_{\mathbb T^2},w(t)]\partial_x w(t)\|_{L^2(\mathbb T^2)}\leq &C_0\Big(\|J^s_{\mathbb T^2}w(t)\|_{L^2(\mathbb T^2)}\|\partial_x w(t)\|_{L^\infty(\mathbb T^2)}\notag\\
 &+(\|w(t)\|_{L^{\infty}(\mathbb T^2)}+\|\nabla w(t)\|_{L^\infty (\mathbb T^2)})\|J^{s-1}_{\mathbb T^2}\partial_x w(t)\|_{L^2(\mathbb T^2)}\Big)\,.\label{ag02_4}
\end{align}
In consequence, from \eqref{ag02_3} and \eqref{ag02_4} we can conclude that
\begin{align}
\frac12\frac d{dt}\|w(t)\|^2_{H^s(\mathbb T^2)}\leq& C_0\|\partial_x w(t)\|_{L^\infty(\mathbb T^2)}\|w(t)\|^2_{H^s(\mathbb T^2)}\notag\\
&+C_0(\|w(t)\|_{L^{\infty}(\mathbb T^2)}
+\|\nabla w(t)\|_{L^\infty (\mathbb T^2)})\|J^{s-1}_{\mathbb T^2}\partial_x w(t)\|_{L^2(\mathbb T^2)}\|J^s_{\mathbb T^2}w(t)\|_{L^2(\mathbb T^2)}\notag\\
&+\frac12\|\partial_xw(t)\|_{L^{\infty}(\mathbb T^2)}\|J^s_{\mathbb T^2}w(t)\|_{L^2(\mathbb T^2)}^2\,.\notag\end{align}
Hence,
\begin{align}
\frac d{dt}\|w(t)\|^2_{H^s(\mathbb T^2)}\leq C_0(\|w(t)\|_{L^{\infty}(\mathbb T^2)}
+\|\nabla w(t)\|_{L^\infty (\mathbb T^2)})\|w(t)\|^2_{H^s(\mathbb T^2)}\,.\label{4.35}
\end{align}
Now, we integrate \eqref{4.35} in $[0,t]$ to obtain
\begin{align*}
 \|w(t)\|^2_{H^s(\mathbb T^2)}&\leq \|w_0\|^2_{H^s(\mathbb T^2)}+C_0\int_0^t (\|w(t')\|_{L^\infty (\mathbb T^2)}+\|\nabla w(t')\|_{L^\infty (\mathbb T^2)})\|w(t')\|^2_{H^s(\mathbb T^2)}dt'\notag\\
 &\leq \|w_0\|^2_{H^s(\mathbb T^2)}+C_0(\|w\|_{L^1_TL^{\infty}(\mathbb T^2)}+
\|\nabla w\|_{L^1_T L^\infty(\mathbb T^2)} )\|w\|^2_{L^\infty_T H^s(\mathbb T^2)}\,,
\end{align*}
for all $t\in[0,T]$. Then, \eqref{ENER} follows from the last inequality.
\end{proof}

The a priori estimate that we will obtain in the following lemma  is based on the Strichartz estimate proved in Corollary \ref{C3-5}. This a priori estimate is essential to guarantee that the approximate solutions to the IVP \eqref{BO}, that we will use in the proof of Theorem \ref{TP}, are defined in a common time interval. 

\begin{lemma}\label{gT} Let $T>0$, $s>\frac74$, $w_0\in H^{\infty}(\mathbb T^2)$, and let $w\in C([0,T];H^{\infty}(\mathbb T^2))$ be a real solution of the IVP \eqref{BOS}. Let us define 
\begin{align} 
g(T):=\|w\|_{L^1_TL^{\infty}(\mathbb T^2)}+
\|\nabla w\|_{L^1_T L^\infty(\mathbb T^2)}\,.\label{NOR}
\end{align}
Then there exist $C_s>0$ such that
\begin{align}
g(T)\leq C_sT^{\frac12}(1+g(T))\|w\|_{L^\infty_T H^s(\mathbb T^2)}\,.\label{EST1}
\end{align}
\end{lemma}
\begin{proof}
Let us observe that $w$, $\partial_xw$, and $\partial_yw$ satisfy the hyphotesis of Corollary \ref{C3-5} with $F(w)$ given by $\frac12w^2$, $\frac12\partial_xw^2$ and $\frac12\partial_yw^2$, respectively. Therefore, for $s':=s-1>\frac34$, using \eqref{L3-5eq3-7} with $F(w):=\frac12w^2$ and Lemma \ref{Pro}, we obtain
\begin{align}
\|w\|_{L^1_T L^\infty(\mathbb T^2)}&\leq CT^{\frac12}(\|w\|_{L^{\infty}_TH^{s'}(\mathbb T^2)}+\|w^2\|_{L^1_TH^{s'}(\mathbb T^2)})\notag\\
&\leq CT^{\frac12}(\|w\|_{L^{\infty}_TH^{s'}(\mathbb T^2)}+\|w\|_{L^{\infty}_TH^{s'}(\mathbb T^2)}\int_0^T\|w(t)\|_{L^{\infty}(\mathbb T^2)}dt)\,.\label{des1}
\end{align}
Now, using \eqref{L3-5eq3-7} with $F(w):=\frac12\partial_xw^2$ and Lemma \ref{Pro}, we obtain
\begin{align}
\|\partial_xw\|_{L^1_T L^\infty(\mathbb T^2)}&\leq CT^{\frac12}(\|\partial_xw\|_{L^{\infty}_TH^{s'}(\mathbb T^2)}+\|\partial_xw^2\|_{L^1_TH^{s'}(\mathbb T^2)})\notag\\
&\leq CT^{\frac12}(\|w\|_{L^{\infty}_TH^{s}(\mathbb T^2)}+\int_0^T\|w^2(t)\|_{H^s(\mathbb T^2)}dt)\notag\\
&\leq CT^{\frac12}(\|w\|_{L^{\infty}_TH^{s}(\mathbb T^2)}+\|w\|_{L^{\infty}_TH^{s}(\mathbb T^2)}\int_0^T\|w(t)\|_{L^{\infty}(\mathbb T^2)}dt).\label{des2}
\end{align}
Last inequality also is true for $\partial_yw$ instead of $\partial_xw$. In this manner, taking into account this observation and inequalities \eqref{des1} and \eqref{des2} we conclude that
\begin{align}
g(T)&\leq CT^{\frac12}(\|w\|_{L^{\infty}_TH^{s}(\mathbb T^2)}+\|w\|_{L^{\infty}_TH^{s}(\mathbb T^2)}\int_0^T\|w(t)\|_{L^{\infty}(\mathbb T^2)}dt)\notag\\
&\leq CT^{\frac12}\|w\|_{L^{\infty}_TH^{s}(\mathbb T^2)}(1+\int_0^T\|w(t)\|_{L^{\infty}(\mathbb T^2)}dt)\leq CT^{\frac12}\|w\|_{L^{\infty}_TH^{s}(\mathbb T^2)}(1+g(T))\,,\notag
\end{align}
which proves \eqref{EST1}.
\end{proof}

\section{Proof of Theorem \ref{TP}}

Using the abstract theory, developed by Kato  in \cite{Ka1975} and \cite{Ka1979}, to prove LWP of the quasi-linear evolutions equations, it can be established the following result of LWP of the IVP (1.1) for initial data in $H^s(\mathbb T^2)$, with $s>2$.

\begin{lemma}\label{L5.1} Let $s>2$ and $w_0\in H^s(\mathbb T^2)$ such that
$$\int_0^{2\pi}w_0(x,y) dx=0\quad \text{a.e. $y\in \mathbb T$}.$$
There exist a positive time $T=T(\|w_0\|_{H^s})$ and a unique solution of the IVP \eqref{BO} in the class
$$C([0,T];H^s(\mathbb T^2))\cap C^1([0,T]; L^2(\mathbb T^2)).$$
Moreover, for any $0<T'<T$, there exists a neighborhood $\mathcal U$ of $w_0$ in $H^s(\mathbb T^2)$ such that the flow map datum-solution
\begin{align*}
S^s_{T'}:\mathcal U&\to C([0,T'];H^s(\mathbb T^2))\\
v_0&\mapsto v,
\end{align*}
is continuous.
\end{lemma}
Let $w_0\in H^\infty(\mathbb T^2)$ and $w$ the solution of the IVP \eqref{BO} of Lemma \ref{L5.1}. Then $w\in C([0,T^*);H^\infty(\mathbb T^2))$, where $T^*$ is the maximal time of existence of $w$ satisfying $T^*\geq T(\|w_0\|_{H^3})$. We have either $T^*=+\infty$ or, if $T^*<\infty$,
\begin{align}
\lim_{t\to T^*} \|w(t)\|_{H^3}=+\infty.\label{5.1}
\end{align}

\begin{lemma}\label{L5.2} Let $\frac 74<s<3$, $C_s$ the constant in \eqref{EST1}, $C_0$ the constant in \eqref{ENER}, $A_s:=8(1+C_0)C_s$ and $T:=(A_s\|w_0\|_{H^s}+1)^{-2}$. Then $T<T^*$,
\[\|w\|_{L^\infty_T H^s(\mathbb T^2)}\leq 2\|w_0\|_{H^s(\mathbb T^2)}\quad\text{and}\quad g(T)\leq \frac 83 C_s\|w_0\|_{H^s(\mathbb T^2)}\]
where $g(T)$ is the norm defined in \eqref{NOR}.
\end{lemma}

\begin{proof} Let $A$ be the set $\{T'\in(0,T^*):\|w\|^2_{L^\infty_{T'} H^s}\leq 4\|w_0\|^2_{H^s}\}$. Since $w\in C([0,T^*);H^\infty(\mathbb T^2))$, the set $A$ is not empty. Let $T_0:=\sup A$. We will prove that $T_0>T$. We argue by contradiction by assuming that $0<T_0\leq T<1$. By continuity we have that $\|u\|^2_{L^\infty_{T_0}H^s(\mathbb T^2)}\leq 4\|u_0\|^2_{H^s(\mathbb T^2)}$.  From \eqref{EST1}, it follows that
\begin{align}
g(T_0)\leq C_s T_0^{1/2}(1+g(T_0))\|w\|_{L^\infty_{T_0}H^s}&\leq 2C_s T_0^{1/2}\|w_0\|_{H^s}(1+g(T_0)) \notag \\
&\leq 2C_s T^{1/2}\|w_0\|_{H^s}(1+g(T_0)) \notag \\
&= 2 C_s\dfrac1{A_s\|w_0\|_{H^s}+1}\|w_0\|_{H^s}(1+g(T_0)).
\end{align}
Hence
$$\left(1-\dfrac{2C_s\|w_0\|_{H^s}}{A_s \|w_0\|_{H^s}+1} \right) g(T_0)\leq \dfrac{2C_s\|w_0\|_{H^s}}{A_s\|w_0\|_{H^s}+1}.$$
Therefore
$$(A_s \|w_0\|_{H^s}+1-2C_s\|w_0\|_{H^s})g(T_0)\leq 2C_s\|w_0\|_{H^s},$$
i.e.,
$$(8(1+C_0)C_s\|w_0\|_{H^s}+1-2C_s\|w_0\|_{H^s})g(T_0)\leq 2C_s\|w_0\|_{H^s},$$
i.e.,
$$(6C_s\|w_0\|_{H^s}+8C_0C_s\|w_0\|_{H^s}+1)g(T_0)\leq 2C_s\|w_0\|_{H^s}.$$
In consequence
$$g(T_0)\leq \dfrac{2C_s\|w_0\|_{H^s}}{6C_s\|w_0\|_{H^s}+8C_0C_s\|w_0\|_{H^s}+1}\leq \dfrac{2C_s}{6C_s+8C_0C_s}=\dfrac1{3+4C_0}\leq \dfrac1{3C_0}.$$
We use the estimate \eqref{ENER} with $s=3$ to obtain
$$\|w\|^2_{L^\infty_{T_0} H^3}\leq \|w_0\|^2_{H^3}+C_0( \| w\|_{L^1_{T_0}L^\infty(\mathbb T^2)}+\|\nabla w\|_{L^1_{T_0}L^\infty(\mathbb T^2)})\|w\|^2_{L^\infty_{T_0} H^3}\leq \|w_0\|^2_{H^3}+C_0 \dfrac1{3C_0}\|w\|^2_{L^\infty_{T_0}H^3}.$$
This way,
$$\|w\|^2_{L^\infty_{T_0} H^3}\leq \dfrac32\|w_0\|^2_{H^3},$$
which implies, taking into account \eqref{5.1}, that $T_0<T^*$.\\

On the other hand, by using the energy estimate \eqref{ENER}, we obtain that
$$\|w\|^2_{L^\infty_{T_0}H^s}\leq \|w_0\|^2_{H^s}+C_0g(T_0)\|w\|^2_{L^\infty_{T_0}H^s}\leq \|w_0\|^2_{H^s}+\dfrac13\|w\|^2_{L^\infty_{T_0}H^s},$$
i.e.,
$$\|w\|^2_{L^\infty_{T_0}H^s}\leq\dfrac32\|w_0\|^2_{H^s},$$
 and by continuity, for some $T'\in(T_0,T^*)$, we have that
$$\|w\|^2_{L^\infty_{T'}H^s}\leq 4\|w_0\|^2_{H^s},$$
i.e., there exists $T'>T_0$, with $T'\in A$. This contradicts the definition of $T_0$. Then we conclude that $T<T_0$, and thus $\|w\|_{L^\infty_{T}H^s}\leq 2\|w_0\|_{H^s}$.\\

From \eqref{EST1} it follows that
$$g(T)\leq C_s T^{1/2}(1+g(T))\|w\|_{L^\infty_{T}H^s}\leq 2C_sT^{1/2}(1+g(T))\|w_0\|_{H^s},$$
hence
$$(1-2C_s T^{1/2}\|w_0\|_{H^s})g(T)\leq 2C_s T^{1/2}\|w_0\|_{H^s}.$$
Let us observe that
\begin{align*}
2C_sT^{1/2}\|w_0\|_{H^s}=\dfrac{2C_s}{A_s\|w_0\|_{H^s}+1}\|w_0\|_{H^s}\leq \dfrac{2C_s}{A_s}\leq \dfrac{2C_s}{8(1+C_0)C_s}<\dfrac14.
\end{align*}
Thus, since $T<1$,
$$\dfrac34g(T)\leq2C_sT^{1/2}\|w_0\|_{H^s}\leq2C_s\|w_0\|_{H^s},$$
and we conclude that
$$g(T)\leq \dfrac 83 C_s\|w_0\|_{H^s},$$
which completes the  proof of Lemma \ref{L5.2}.
\end{proof}

\subsection{Sketch of the proof of Theorem 1.1} \text{}\\

The most important tools in this proof are the results contained in Lemmas \ref{EE} and \ref{L5.2}. \\
Given $w_0\in H^s(\mathbb T^2)$, with $\frac 74<s<3$, we will use the Bona-Smith argument (see \cite{BS1975}), regularizing the initial datum $w_0$ as follows. Let $\rho\in C^\infty_0(\mathbb R)$ such that $0\leq\rho\leq1$, $\rho(x)=0$ if $|x|>1$, and $\rho(x)=1$ if $|x|<1/2$. We define $\widetilde \rho(x,y):=\rho(\sqrt{x^2+y^2})$ and, for each $n\in\mathbb N$ we define $w_{0n}$ by
$$\widehat {w_{0,n}}(m',n'):=\widetilde\rho\left(\frac {m'} n,\frac {n'} n\right)\widehat {w_0} (m',n'), \ \textnormal{for} \  (m',n')\in\mathbb Z^2.$$\\
It can be seen that, for each $n\in\mathbb N$, $w_{0,n}\in H^\infty(\mathbb T^2)$ and that  $w_{0,n}\overset{n\to\infty}\longrightarrow w_0$ in $H^s(\mathbb T^2)$. Now, for each $n\in\mathbb N$, we consider the solution $w_n$ of the IVP associated to the equation in (1.1) with initial datum $w_{0,n}$. The existence of the solutions $w_n$ is guaranteed by Lemma 5.1. \\
From Lemma \ref{L5.2} we have that $w_n\in C([0,T_n];H^\infty(\mathbb T^2))$, where $T_n:=(A_s\|w_{0,n}\|_{H^s}+1)^{-2}$. Since $w_{0,n}\overset{n\to\infty}\longrightarrow w_0$ in $H^s(\mathbb T^2)$, there exists $N\in\mathbb N$ such that for all $n\geq N$, $T_n\geq (2A_s\|w_0\|_{H^s}+1)^{-2}$. In consequence, for all $n\geq N$, $w_n\in C([0,T];H^\infty(\mathbb T^2))$, where $T:=(2A_s\|u_0\|_{H^s}+1)^{-2}$. Without loss of generality we assume that, for each $n\in\mathbb N$, $w_n\in C([0,T];H^\infty(\mathbb T^2))$. Besides, from Lemma \ref{L5.2}, we can suppose that, for each $n\in\mathbb N$,
\begin{align}
\|w_n\|_{L^\infty_T H^s(\mathbb T^2)}\leq 4\|w_0\|_{H^s(\mathbb T^2)},\label{5.2}
\end{align}
and
\begin{align}
\|w_n\|_{L^1_TL^\infty_{xy}}+\|\nabla w_n\|_{L^1_T L^\infty_{xy}}\leq \frac 83C_s\|w_{0,n}\|_{H^s}\leq\frac{16}3 C_s\|w_0\|_{H^s}\equiv \widetilde K.\label{5.3}
\end{align}
From \eqref{5.2} it follows that the sequence $\{w_n\}$ is bounded in $L^\infty([0,T];H^s(\mathbb T^2))$. Therefore, there exist a subsequence of $\{w_n\}$, which we continue denoting by $\{w_n\}$, and a function $w\in L^\infty([0,T];H^s(\mathbb T^2))$ such that $w_n \overset *\rightharpoonup w$ in $L^\infty([0,T];H^s(\mathbb T^2))$, when $n\to\infty$ (weak $\ast$ convergence in $L^\infty([0,T];H^s(\mathbb T^2))$). \\
It can be proved in analogous form as it was done in \cite{BJM2019} that  $w\in C([0,T];H^s(\mathbb T^2))$, with $w$, $w_x$, $w_y$ in $L^1([0,T];L^\infty_{xy})$, and that $w$ is the solution of the IVP (1.1).\\
The uniqueness and the continuous dependence on the initial data also follow in a similar way as in \cite{BJM2019}.\qed
\bigskip

\textbf{Acknowledgments}\\

Supported by Facultad de Ciencias, Universidad Nacional de Colombia, Sede Medell\'in, project ``Ecuaciones Diferenciales no Lineales", Hermes code 44342.


\begin{thebibliography}{12}
\bibitem{BS1975} Bona J., Smith R., \textit{The initial value problem for the Korteweg-de Vries equation}, Philos. Trans. R. Soc. Lond., Ser. A., {\bf278} (1975), 555-601.
\bibitem{BP2008} Burq, N., Planchon, F., \textit{On well-posedness for the Benjamin-Ono equation}, Mathematische Annalen, {\bf 340} (2008), No. 3, 497-542.
\bibitem{BJM2019} Bustamante E., Jim\'enez Urrea J., Mej\'ia J., \textit{The Cauchy problem for a family of two-dimensional fractional Benjamin-Ono equations}, Communications on Pure and Applied Analysis, {\bf 18} (2019), 1177--1203.
\bibitem{EP2018} Esfahani, A., Pastor, A. \textit{Two dimensional solitary waves in shear flows}, Calc. Var. Partial Differential Equations, {\bf{57}} (2018), No. 4, 57:102.
\bibitem{G2008} Grafakos, L., \textit{Classical Fourier Analysis}, 2nd ed, Springer, New York, 2008.
\bibitem{IK2007} Ionescu A. D., Kenig C. E., \textit{Global well-posedness of the Benjamin Ono in low-regularity spaces}, Journal of the American Mathematical Society, {\bf20} (2007), No. 3, 755-798.
\bibitem{IK20071} Ionescu A. D., Kenig C. E., \textit{Local and global well-posedness of periodic KP-I equations}, Ann. of Math. Stud., 163, 181-211, 2007.
\bibitem{I1986} Iorio, J., \textit{On the Cauchy problem for the Benjamin-Ono equation}, Comm. Partial Differential Equations, Vol. {\bf 11}, No. 11 (1986), 1031-1081.
\bibitem{Ka1975} T. Kato, {Quasilinear equations of evolution, with applications to PDE}, \emph{Lecture Notes in Mathematics}, {\bf448}, Springer, Berlin, (1975), 25--70.
\bibitem{Ka1979} T. Kato, {On the Korteweg-de Vries equation}, \emph{Manuscripta Math.}, {\bf 28} (1979), 89-99.
\bibitem{K2004} Kenig, C., \textit{On the local and global well-posedness theory for the KP-I equation}, Annales de l'Institute Henri Poincare (C) Non Linear Analysis, Vol. {\bf 21}, No. 6 (2004), 827-838.
\bibitem{KeKoe2003} Kenig, C., Koenig, K.D., \textit{On the local well posedness of the Benjamin-Ono and modified Benjamin-Ono equations}, Math. Res. Lett., {\bf 10} (2003), 879-895.
\bibitem{KT2003} Koch, H., Tzvetkov, N., \textit{On the local well-posedness of the Benjamin-Ono equation in $H^s(\mathbb R)$}, {\bf IMRN} International Mathematics Research Notices, (2003), No. 26, 1449-1464.
\bibitem{LPRT} Linares F., Panthee M., Robert T.,  Tzvetkov N., \textit{On the periodic Zakharov-Kuznetsov equation}, arXiv:1809.02027 [math.AP], 6 sep 2018.
\bibitem{LPS2014} Linares, F., Pilod, D., Saut, J.C., \textit{Dispersive perturbations of Burgers and hyperbolic equations I: Local theory}, SIAM J. Math. Anal., vol {\bf 46}, No. 2, (2014), 1505-1537.
\bibitem{LPS2017} Linares, F., Pilod, D., Saut, J.C., \textit{The Cauchy problem for the fractional Kadomtsev-Petviashvili equations}, SIAM J. Math Anal., vol {\bf 50}, No. 3 (2018), 3172-3209.
\bibitem{M2007} Molinet, L., \textit{Global well-posedness in the energy space for the Benjamin-Ono equation on the circle}, Mathematische Annalen, Vol. {\bf 337}, No. 2 (2007), 353-383.
\bibitem{M2008} Molinet, L., \textit{Global well-posedness in $L^2$ for the periodic Benjamin-Ono equation}, American Journal of Mathematics, Vol. {\bf130}, No. 3 (2008), 635-683.
\bibitem{MR2008} Molinet, L., Ribaud, F., \textit{Well-posedness in $H^1$ for generalized Benjamin-Ono equations on the circle}, Discrete and Continuous Dynamical Systems, vol {\bf 23}, No. 4 (2008), 1295-1311.
\bibitem{N1996} Nathanson, M., \textit{Aditive number theory the classical bases}, Springer, New York, 1996.
\bibitem{PS1995} Pelinovsky, D.E., Shrira,V.I., \textit{Collapse transformation for self-focusing solitary waves in boundary-layer type shear flows}, Physics Letters A, {\bf 206} (1995), 195-202.
\bibitem{P1991} Ponce, G., \textit{On the global well-posedness of the Benjamin-Ono equation}, Differential Integral Equations, {\bf {4}} (1991), No. 3, 527-542. 
\bibitem{T2004} Tao, T., \textit{Global well-posedness of the Benjamin-Ono equation in $H^1(\mathbb R)$}, Journal of Hyperbolic Differential Equations, vol {\bf 1}, No. 1, (2004), 27-49.
\bibitem{Tao} Tao, T., \textit{Nonlinear Dispersive Equations. Local and Global Analysis}, Regional Conference Series in Mathematics, Number 106, AMS, (2006).







\end{thebibliography}
\end{document}